\documentclass[a4paper, 10pt]{amsart}
\usepackage{amsmath}
\usepackage{amsfonts,amssymb,color,hyperref}
\usepackage{enumerate}
\usepackage{multicol}
\usepackage{verbatim}
\usepackage{url}
\usepackage{tikz}
\usepackage{setspace}
\usepackage[utf8]{inputenc} 
\usepackage[T1]{fontenc}
\usetikzlibrary{arrows}

\theoremstyle{plain}
\newtheorem{theorem}{\bf Theorem}[section]
\newtheorem{proposition}[theorem]{\bf Proposition}
\newtheorem{lemma}[theorem]{\bf Lemma}

\theoremstyle{definition}

\newcommand{\Z}{\mathbb Z}

\renewcommand{\v}{{\sf v}}

\newcommand{\bdot}{\boldsymbol{\cdot}}

 \DeclareMathOperator{\ord}{ord}
\DeclareMathOperator{\lcm}{lcm} 

 \DeclareMathOperator{\supp}{supp}

\renewcommand{\t}{\, | \,}

\numberwithin{equation}{section}

\subjclass[2010]{11B75 (primary), 20D60, 11P70 (secondary)}
\title{Some zero-sum problems over $\langle x,y \mid x^2 = y^{n/2}, y^n = 1, yx = xy^s \rangle$}
\keywords{Zero-sum problems, Gao conjecture, Zhuang-Gao conjecture, groups with a cyclic index 2 subgroup}

\author[S. Ribas]{S. Ribas}
\address{Departamento de Matem\'{a}tica\\
Universidade Federal de Ouro Preto\\
Ouro Preto, MG\\
35402-136\\
Brazil\\
}
\email{savio.ribas@ufop.edu.br }

\date{\today}

\begin{document}

\maketitle

\begin{abstract}
Let $n \ge 8$ be even, and let $G = \langle x, y \mid x^2 = y^{n/2}, y^n = 1, yx = xy^s \rangle$, where $s^2 \equiv 1 \pmod n$ and $s \not\equiv \pm1 \pmod n$. In this paper, we provide the precise values of some zero-sum constants over $G$, namely the small Davenport constant, $\eta$-constant, Gao constant, and Erd\H os-Ginzburg-Ziv constant. In particular, the Gao's and Zhuang-Gao's Conjectures hold for $G$. We also solve the associated inverse problems when $n \equiv 0 \pmod 4$. 
\end{abstract}

\section{Introduction}

%

Let $G$ be a finite group multiplicatively written. The {\em zero-sum problems} consist of establishing conditions that guarantee that a given sequence over $G$ has a non-empty product-one subsequence with some prescribed property. These problems date back to Erd\H os, Ginzburg and Ziv \cite{EGZ}, van Emde Boas and Kruyswijk \cite{vEBK} and Olson \cite{Ols1,Ols2}. It has applications and connections in several branches of mathematics. Among others, we highlight the factorization theory \cite{GeHK}. For an overview over abelian groups, see the surveys from Caro \cite{Car} and Gao and Geroldinger \cite{GaGe}. Later, the zero-sum problems were further generalized to non-abelian settings \cite{Yu,YP,ZhGa}. For more recent results, see \cite{Bas,MLMR,MR1,MR2,MR3,FZ,GLQ,GGOZ,HZ,OZ1,QL1,QL2}.

\subsection{Definitions and notations}

Let $\mathcal F(G)$ be a free abelian monoid with basis $G$. A sequence $S$ over $G$ is an element of $\mathcal F(G)$. Thus sequences are formed by finitely many terms from $G$ where repetition is allowed and order is disregarded. The operation in $\mathcal F(G)$ is the sequence concatenation  product denoted by ${\bdot}$. Therefore a sequence $S \in \mathcal F(G)$ has the form 
$$S = g_1 {\bdot} \dots {\bdot} g_k = \prod_{1\le i \le k}^{\bullet} \, g_i = \prod_{1\le i \le k}^{\bullet}
g_{\tau(i)} = \prod_{g \in G}^{\bullet} g^{[\v_g(S)]}$$ 
for any permutation $\tau$ of $[1,k]$, where $g_1, \dots, g_k \in G$ are the terms of $S$, $\v_g(S)$ is the multiplicity of $g$ in $S$, and $|S| = k = \sum_{g \in G} \v_g(G) \ge 0$ is the length of $S$. A sequence $T$ is a subsequence of $S$ if $T \mid S$ as elements of $\mathcal F(G)$; equivalently, if $\v_g(T) \le \v_g(S)$ for all $g \in G$. In this case, we write $S {\bdot} T^{[-1]} = \prod_{g \in G}^{\bullet}  g^{[\v_g(S) - \v_g(T)]}$. Furthermore, denote $\prod_{1\le i \le k}^{\bullet} T$ by $T^{[k]}$ and let $S_K = \prod_{g \in K}^{\bullet} g^{[\v_g(S)]}$ be the subsequence of $S$ formed by the terms that lie in a subset $K \subset G$.


The set of products of $S$, the set of subproducts of $S$, and the set of $n$-subproducts of $S$ are defined, respectively, by
$$\pi(S) = \left\{ \prod_{i=1}^k g_{\tau(i)} \in G \mid \; \tau \text{ is a permutation of $[1,k]$}\right\}, \quad \Pi(S) = \bigcup_{T \mid S \atop |T| \ge 1} \pi(T), \quad \text{and} \quad \Pi_n(S) = \bigcup_{T \mid S \atop |T| = n} \pi(T).$$

The sequence $S$ is said to be a {\em product-one sequence} if $1 \in \pi(S)$, an {\em $n$-product-one sequence} if $S$ is a product-one sequence of length $|S| = n$, a {\em short product-one sequence} if $S$ is a product-one sequence with $1 \le |S| \le \exp(G)$, where $\exp(G) = \lcm\{\ord(g); \; g \in G\}$ denotes the exponent of $G$, a {\em product-one free sequence} if $1 \not\in \Pi(S)$, an {\em $n$-product-one free sequence} if $1 \not\in \Pi_n(S)$, and a {\em short product-one free sequence} if $1 \not\in \prod_{j}(S)$ for every $j \in [1,\exp(G)]$. For a normal subgroup $H \vartriangleleft G$, let $\phi_H \colon G\rightarrow G/H$ be the natural epimorphism. Let $T \in \mathcal F(G)$ with $|T| = n$. If $\phi_H(T) \in \mathcal F(G/H)$ is a product-one sequence, then we say $T$ is an {\em $n$-product-$H$ sequence}. The case $H = \{1\}$ means that $T$ is an $n$-product-one sequence.

The following zero-sum invariants are defined:
\begin{enumerate}[(a)]
\item {\em Small Davenport constant}, ${\sf d}(G)$, is the maximal length of all product-one free sequences over $G$,
\item {\em Gao constant}, ${\sf E}(G)$, is the smallest $\ell > 0$ such that every $S \in \mathcal F(G)$ with $|S| \ge \ell$ has a $|G|$-product-one subsequence,
\item {\em $\eta$-constant}, $\eta(G)$, is the smallest $\ell > 0$ such that every $S \in \mathcal F(G)$ with $|S| \ge \ell$ has a short product-one subsequence,
\item {\em Erd\H os-Ginzburg-Ziv constant}, ${\sf s}(G)$, is the smallest $\ell > 0$ such that every $S \in \mathcal F(G)$ with $|S| \ge \ell$ has an $\exp(G)$-product-one subsequence.
\end{enumerate}

For a finite group $G$, all these invariants are well-defined and finite.

\subsection{Main results}

Throughout this paper, $n \ge 8$ is an even positive integer which is not twice an odd prime power and $s$ is an integer satisfying $s^2 \equiv 1 \pmod n$ and $s \not\equiv \pm1 \pmod n$ (in fact, if $n$ is twice an odd prime power, then such $s$ does not exist). Moreover, let 
\begin{equation}\label{def:G}
G = \langle x,y: x^2=y^{n/2}, \; y^n=1, \; yx=xy^s \rangle,
\end{equation}
for which 
\begin{equation}\label{expG}
\exp(G) = 
\begin{cases}
n &\text{ if $n \equiv 0 \pmod 4$,} \\
2n &\text{ if $n \equiv 2 \pmod 4$.}
\end{cases}
\end{equation}
This group naturally extends the group studied in \cite{AMR}, as well as dicyclic groups extend dihedral groups. 

The main result of this paper is the following.

\begin{theorem}\label{thm:main}
Let $G$ be the group defined by \eqref{def:G}. Then 
$${\sf d}(G) = n, \quad
\eta(G) = n+1, \quad
{\sf s}(G) = 
\begin{cases} 
2n &\text{ if $n \equiv 0 \pmod 4$,} \\
3n &\text{ if $n \equiv 2 \pmod 4$,}
\end{cases} \quad \text{ and } \quad
{\sf E}(G) = 3n.$$
\end{theorem}

The proof splits into the cases $n \equiv 0 \pmod 4$ and $n \equiv 2 \pmod 4$. In each of them we use the inductive method with distinct normal subgroups. The problems related to finding the exact values of the invariants are called {\em direct problems}. After solving the direct problems, it is natural to ask the respective {\em inverse problems}, which consist of finding the structure of ($|G|$-, $\exp(G)$-, short) product-one free sequences. For $n \equiv 0 \pmod 4$, we further establish the related inverse problems.

\begin{theorem}\label{thm:inverse0}
Let $G$ be the group defined by \eqref{def:G}, where $n \equiv 0 \pmod 4$, and let $S \in \mathcal F(G)$.
\begin{enumerate}[(a)]
\item If $|S| = {\sf s}(G) - 1 = 2n - 1$, then $S$ is $n$-product-one free if and only if there exist $\alpha, \beta \in G$ and $t_1, t_2, t_3 \in \mathbb Z$ such that $G \cong \langle \alpha, \beta \mid \alpha^2 = \beta^{n/2}, \beta^n = 1, \beta \alpha = \alpha \beta^s \rangle$, $\gcd(t_1-t_2, n) = 1$ and $S = (\beta^{t_1})^{[n-1]} \bdot (\beta^{t_2})^{[n-1]} \bdot (\alpha \beta^{t_3})$.

\item If $|S| = {\sf E}(G) - 1 = 3n - 1$, then $S$ is $2n$-product-one free if and only if there exist $\alpha, \beta \in G$ and $t_1, t_2, t_3 \in \mathbb Z$ such that $G \cong \langle \alpha, \beta \mid \alpha^2 = \beta^{n/2}, \beta^n = 1, \beta \alpha = \alpha \beta^s \rangle$, $\gcd(t_1-t_2, n) = 1$ and $S = (\beta^{t_1})^{[2n-1]} \bdot (\beta^{t_2})^{[n-1]} \bdot (\alpha \beta^{t_3})$.

\item If $|S| = {\sf \eta}(G) - 1 = {\sf d}(G) = n$, then $S$ is (short) product-one free if and only if there exist $\alpha, \beta \in G$ and $t \in \mathbb Z$ such that $G \cong \langle \alpha, \beta \mid \alpha^2 = \beta^{n/2}, \beta^n = 1, \beta \alpha = \alpha \beta^s \rangle$ and $S = \beta^{[n-1]} \bdot \alpha \beta^t$.
\end{enumerate}
\end{theorem}

The paper is organized as follows. In Section \ref{sec:pre}, we present some well-known results that will be used throughout the paper. In Section \ref{sec:direct}, we prove the direct problems, culminating in the proof of Theorem \ref{thm:main}. In Section \ref{sec:inverse}, we prove the inverse problems (Theorem \ref{thm:inverse0}).

\section{Preliminary results}\label{sec:pre}

In this section, we present some well-known results that will be used throughout the paper. The first one deals with a direct and its respective inverse problem related to $C_n = \langle y \mid y^n = 1 \rangle$, the cyclic group of order $n$.

\begin{lemma}\label{lem:inverseegzcyclic}
\begin{enumerate}[(i)]
\item For every $n \ge 2$, ${\sf s}(C_n) = {\sf E}(C_n) = 2n-1$.
\item Let $n \ge 2$ be an integer and let $2 \le k \le \lfloor n/2 \rfloor + 2$. Suppose that $S \in \mathcal F(C_n)$ and $|S| = 2n-k$. If $S$ is $n$-product-one free, then there exists $a \bdot b \mid S$ such that $\min\{{\sf v}_a(S), {\sf v}_b(S)\} \ge n - 2k + 3$ and ${\sf v}_a(S) + {\sf v}_b(S) \ge 2n - 2k + 2$, where $C_n = \langle ab^{-1} \rangle$.
\end{enumerate}
\end{lemma}

\begin{proof}
For {\it (i)}, see \cite{EGZ}. For {\it (ii)}, see \cite{Ga1}.
\end{proof}

The second deals with a direct and its respective inverse problem related to $Q_{4n} = \langle x,y \mid x^2 = y^n, y^{2n} = 1, yx = xy^{-1} \rangle$, the dicyclic group of order $4n$.

\begin{lemma}\label{lem:dicyclic}
\begin{enumerate}[(i)]
\item For every $n \ge 2$, ${\sf E}(Q_{4n}) = 6n$.

\item Let $n \ge 3$ and let $S\in \mathcal F(Q_{4n})$ with $|S|=6n-1$. Then $S$ has no $4n$-product-one subsequence if and only if there exist $\alpha,\beta\in Q_{4n}$ with $Q_{4n}=\langle \alpha,\beta\mid \alpha^2=\beta^{n/2}, \beta^n=1, \beta\alpha=\alpha\beta^{-1}\rangle$ such that $S=(\beta^{t_1})^{[4n-1]}\bdot (\beta^{t_2})^{[2n-1]}\bdot \alpha\beta^{t_3}$, where $t_1,t_2,t_3\in [0,2n-1]$ with $\gcd(t_1-t_2,n)=1$.
		
\end{enumerate}
\end{lemma}

\begin{proof}
For {\it (i)}, see \cite[Theorem 10]{Bas}. For {\it (ii)}, see \cite[Theorem 1.3]{OZ1}.
\end{proof}

We now state two direct problems related to the group $C_2 \times C_{2n}$.

\begin{lemma}\label{lem:abelianrank2s}
For every $n \ge 1$, ${\sf s}(C_2 \oplus C_{2n}) = 4n + 1$ and ${\sf E}(C_2 \oplus C_{2n}) = 6n$.

\end{lemma}

\begin{proof}
For the first equation, see \cite[Theorem 5.8.3]{GeHK}. 
The second one is just a combination of \cite[Theorem 1]{Ga2} and \cite[Theorem 1]{Ols2}.
\end{proof}



The proof of Theorem \ref{thm:main} is widely simplified using the following result, since it will be enough obtaining good upper bounds for ${\sf E}(G)$ and ${\sf s}(G)$, as well as good lower bounds for ${\sf d}(G)$ and $\eta(G)$.

\begin{lemma}\label{lem:ineqs}
For every finite group $G$, ${\sf E}(G) \ge {\sf d}(G) + |G|$ and ${\sf s}(G) \ge \eta(G) + \exp(G) - 1$.
\end{lemma}

\begin{proof}
See \cite[Lemma 4]{ZhGa}.
\end{proof}

In \cite{Ga2}, Gao conjectured that ${\sf s}(G) = \eta(G) + \exp(G) - 1$ for every finite group $G$, and in \cite{ZhGa}, Zhuang and Gao conjectured that ${\sf E}(G) = {\sf d}(G) + |G|$ for every finite group $G$. The latter was proven for abelian groups in \cite{Ga2}.

\section{The direct problems}\label{sec:direct}

In this section, we prove Theorem \ref{thm:main}. We first deal with the lower bounds for ${\sf d}(G)$ and for $\eta(G)$.

\begin{lemma}\label{lem:daveta}
Let $G$ be the group defined by \eqref{def:G}. Then ${\sf d}(G) \ge n$ and $\eta(G) \ge n+1$.
\end{lemma}

\begin{proof}
Both inequalities follow from the fact that the sequence $y^{[n-1]} \bdot x \in \mathcal F(G)$ is $n$-product-one free.
\end{proof}

Next we deal with upper bounds for the case $n \equiv 0 \pmod 4$.

\begin{proposition}\label{prop:dicyclic0}
Let $G$ be the group defined by \eqref{def:G}, where $n \equiv 0 \pmod 4$. Then ${\sf s}(G) \le 2n$ and ${\sf E}(G) \le 3n$.
\end{proposition}

\begin{proof}
If ${\sf s}(G) \le 2n$, then any sequence over $G$ of length $3n$ contains two disjoint $n$-product-one subsequences. Therefore ${\sf E}(G) \le 3n$. Hence it suffices to show that any sequence over $G$ of length $2n$ contains an $n$-product-one subsequence. For this, let $S \in \mathcal F(G)$ of length $|S| = 2n$.

Let $L = \langle y^2 \rangle \cong C_{\frac{n}{2}}$, which is a normal subgroup of $G$, and decompose $S = S_L \bdot S_{yL} \bdot S_{xL} \bdot S_{xyL}$. The product of any two terms in the same subsequence belongs to $L$. Therefore it is possible to extract from $S$ at least 
\begin{align*}
\left\lfloor \frac{|S_L|}{2} \right\rfloor + \left\lfloor \frac{|S_{yL}|}{2} \right\rfloor + \left\lfloor \frac{|S_{xL}|}{2} \right\rfloor + \left\lfloor \frac{|S_{xyL}|}{2} \right\rfloor &\ge \frac{|S_L| - 1}{2} + \frac{|S_{yL}| - 1}{2} + \frac{|S_{xL}| - 1}{2} + \frac{|S_{xyL}| - 1}{2} \\
&= \frac{2n - 4}{2} = n - 2 = {\sf s}(L) - 1
\end{align*}
$2$-product-$L$ disjoint subsequences. If either the later inequality is strict or these $n-2$ products over $L$ form a sequence that is not $\frac{n}{2}$-product-one free, then we are done. Otherwise, $S = U_1 \bdot \dots \bdot U_{n-2} \bdot U_0$, where $U_0$ has no $2$-product-$L$ subsequences, $|U_0| = 4$ and $|U_i| = 2$ with $\pi(U_i) \subset L$ for every $i \in [1,n-2]$. Since $n-2 = {\sf s}(C_{\frac{n}{2}}) - 1$, by Lemma \ref{lem:inverseegzcyclic}, if $g_i \in \pi(U_i)$ for $i \in [1,n-2]$, then we may assume, after renumbering if necessary, that $g_i = y^{2a}$ for $i \in [1,\frac{n}{2}-1]$ and $g_i = y^{2b}$ for $i \in [\frac{n}{2},n-2]$, where $\gcd(a-b, \frac{n}{2}) = 1$. Notice that $U_0$ is a $4$-product-$L$ sequence. Let $y^{2c} \in \pi(U_0)$ and let $\ell \in [0,\frac{n}{2}-1]$ such that $\ell \equiv (a-b)^{-1}(2b-c) \pmod {\frac{n}{2}}$. Hence 
$$U_0 \bdot U_1 \bdot \dots \bdot U_{\ell} \bdot U_{\frac{n}{2}} \bdot \dots \bdot U_{n-\ell-3}$$
is an $n$-product-one subsequence of $S$.
\end{proof}

Finally, we deal with upper bounds for the case $n \equiv 2 \pmod 4$.

\begin{proposition}\label{prop:dicyclic2}
Let $G$ be the group defined by \eqref{def:G}, where $n \equiv 2 \pmod 4$. Then ${\sf s}(G)  = {\sf E}(G) \le 3n$.
\end{proposition}

\begin{proof}
By \eqref{expG}, we have that $\exp(G) = |G|$, thus ${\sf s}(G) = {\sf E}(G)$. Let $S \in \mathcal F(G)$ of length $|S| = 2n$. It remains to show that $S$ has an $n$-product-one subsequence.

Let $n_1 = \gcd(n, s+1)$ and $n_2 = n/n_1$, so that $n = n_1n_2$, $n_1$ is even, $s \equiv -1 \pmod {n_1}$ and $s \equiv 1 \pmod {n_2}$. Since $n \equiv 2 \pmod 4$ and $n_1$ is even, it follows that $n_2$ is odd. Moreover, since $s$ is odd, we may assume that $3 \le s \le n-3$ and we further have $s \equiv 1 \pmod {2n_2}$. Set 
$$H = \langle x, y^{n_2} \rangle \cong Q_{2n_1} \quad \text{ and } \quad K = 
\langle y^{2n_2} \rangle \cong C_{\frac{n_1}{2}}.$$ 
Then both $H$ and $K$ are normal subgroups of $G$, whence 
$$G/H \cong C_{n_2} \quad \text{ and } \quad G/K \cong 
C_2 \oplus C_{2n_2}.$$ 

Assume to the contrary that $S$ has no $n$-product-one subsequence. Since ${\sf E}(G/H) = {\sf E}(C_{n_2}) = 2n_2-1$, it is possible to decompose $S = T_1 \bdot \ldots \bdot T_{2n_1-1} \bdot T_0$, where $T_i$ is an $n_2$-product-$H$ subsequence for every $i \in [1,2n_1-1]$ and $T_0$ is a subsequence of $S$ of length $|T_0| = n_2$. Let $h_i\in \pi(T_i)$ for each $i\in [1,2n_1-1]$. Then $h_1\bdot \ldots\bdot h_{2n_1-1} \in \mathcal F(H)$ has no product-one subsequence of length $n_1$. 

Suppose that $n_1 = 2$. Since $\frac{n}{2} = n_2$ divides $s-1$ and $s-1 \le n-4$, it follows that $s = n_2+1$. Since $s$ is odd, we have that $n_2$ is even, a contradiction. Moreover, $n_1 = 4$ implies that $n \equiv 0 \pmod 4$, another contradiction. From now on, we assume that $n_1 \ge 6$.

By Lemma \ref{lem:dicyclic}, it follows that $h_1\bdot \ldots\bdot h_{2n_1-1} = (y^{t_1n_2})^{[n_1-1]} \bdot (y^{t_2n_2})^{[n_1-1]} \bdot  xy^{t_3n_2}$, where $t_1,t_2,t_3\in [0,n_1-1]$ with $\gcd(t_1-t_2,n_1)=1$. After renumbering if necessary, we may assume that $\pi(T_i) = \{y^{t_1n_2}\}$ for $i\in [1,n_1-1]$, $\pi(T_j)=\{y^{t_2n_2}\}$ for $j\in [n_1,2n_1-2]$, and $\pi(T_{2n_1-1}) = \{xy^{t_3n_2}\}$. 

Notice that $|T_{n_1-3} \bdot T_{n_1-2} \bdot T_{n_1-1} \bdot T_{2n_1-2} \bdot T_{2n_1-1} \bdot T_0| = 6n_2 = {\sf E}(C_2 \oplus C_{2n_2}) = {\sf E}(G/K)$, therefore $T_{n_1-3} \bdot T_{n_1-2} \bdot T_{n_1-1} \bdot T_{2n_1-2} \bdot T_{2n_1-1} \bdot T_0$ contains a $4n_2$-product-$K$ subsequence $W$. Let $y^{2tn_2}\in \pi(W) \subset K$, where $t\in [0,\frac{n_1}{2}-1]$. If 
\begin{equation}\label{eq:k}
2tn_2 + (n_1 - 4 - k)t_1n_2 + kt_2n_2 \equiv 0 \pmod n
\end{equation} 
for some $k \in [0, n_1-4]$, then 
$$W \bdot T_1 \bdot \ldots \bdot T_{n_1-4-k} \bdot T_{n_1} \bdot \ldots \bdot T_{n_1+k-1}$$ 
is an $n$-product-one sequence and we are done. But \eqref{eq:k} is equivalent to 
\begin{equation}\label{eq:equivk}
k(t_2-t_1) \equiv 4t_1-2t \pmod {n_1}, 
\end{equation}
which has a solution $k \in [0,n_1-1]$ since $\gcd(t_1-t_2,n_1)=1$. By \eqref{eq:equivk}, $k$ must be even. Thus the only remaining case is $k = n_1-2$, which implies that $2t_1 + 2t_2 \equiv 2t \pmod {n_1}$. In particular, $\pi(W) = \{y^{2tn_2}\}$. 

Since $|W| = 4n_2$ and $G/H \cong C_{n_2}$, it is possible to decompose $W = W_1 \bdot W_2 \bdot W_3 \bdot W_4$, where $W_i$ is an $n_2$-product-$H$ subsequence for $i \in [1,4]$. We may assume without loss of generality that $\pi(W_1) = \pi(W_2) = \pi(W_3) = y^{t_1n_2}$ and $\pi(W_4) = y^{t_2n_2}$. It implies that $y^{(3t_1+t_2)n_2} \in \pi(W) = \{y^{2tn_2}\}$, that is, $3t_1+t_2 \equiv 2t \pmod {n_1}$, a contradiction since $3t_1+t_2$ is odd.
\end{proof}

\subsection{Proof of Theorem \ref{thm:main}}
We are now able to solve the direct problems. From Lemma \ref{lem:daveta}, Lemma \ref{lem:ineqs}, Proposition \ref{prop:dicyclic0} and Proposition \ref{prop:dicyclic2}, it follows that
$$n \le {\sf d}(G) \le {\sf E}(G) - |G| \le 3n - 2n = n$$
and either
$$n+1 \le \eta(G) \le {\sf s}(G) - \exp(G) + 1 \le 2n - n + 1 = n+1 \quad \text{ for $n \equiv 0 \pmod 4$,}$$
or
$$n+1 \le \eta(G) \le {\sf s}(G) - \exp(G) + 1 \le 3n - 2n + 1 = n+1 \quad \text{ for $n \equiv 2 \pmod 4$.}$$
Thus we are done.
\qed

\section{The inverse problem for $n \equiv 0 \pmod 4$}\label{sec:inverse}

In this section, we prove the inverse problem for $n \equiv 0 \pmod 4$. We need the following result.

\begin{lemma}\label{lem:distinct4prod}
Let $G$ be the group defined by \eqref{def:G}, where $n \equiv 0 \pmod 4$, $T = T_1 \bdot \dots \bdot T_{n-3} \in \mathcal F(G)$ such that $|T_i| = 2$ for every $i \in [1,n-3]$, $|T_0| = 4$ and $\{y^{2w}, y^{2w+\frac{n}{2}}\} \subset \pi(T_0)$. Suppose that 
$$y^{2a} \in \pi(T_i) \text{ for } i \in \left[ 1, \frac{n}{2}-1 \right] \quad \text{ and } \quad y^{2b} \in \pi(T_i) \text{ for } i \in \left[ \frac{n}{2}, n-3 \right],$$
where $\gcd\left(a-b, \frac{n}{2}\right) = 1$. Then $1 \in \Pi_n(T \bdot T_0)$. More specifically, $T \bdot T_0$ contains an $n$-product-one subsequence $T'$ with $T_0 \mid T'$.
\end{lemma}

\begin{proof}
This is just an adaptation of \cite[Lemma 2.6]{AMR}.
\end{proof}

\subsection{Proof of Theorem \ref{thm:inverse0}}

We are now able to solve the inverse problem for $n \equiv 0 \pmod 4$.

\begin{enumerate}[{\it (a)}]
\item Let $S \in \mathcal F(G)$ be a $2n$-product-one free sequence of length $|S| = 2n-1$. Similarly as in Proposition \ref{prop:dicyclic0}, it is possible to decompose $S = U_1 \bdot \dots \bdot U_{n-2} \bdot  U_0$ with $|U_i| = 2$ and $\pi(U_i) \subset L = \langle y^2 \rangle$ for every $i \in [1,n-2]$, and $|U_0| = 3$. Let $h_i \in \pi(U_i)$ for $i \in [1,n-2]$. It follows from Lemma \ref{lem:inverseegzcyclic} that 
\begin{equation}\label{h1hn}
T := h_1 \bdot \dots \bdot  h_{n-2}=(y^{2a})^{[\frac n2 -1]} \bdot (y^{2b})^{[\frac n2 -1]}, 
\end{equation}
where $\gcd\left(a-b, \frac n2 \right) = 1$, and 
$$T_0 = (y^{2\alpha})^{[\varepsilon]} \bdot (y^{2\beta+1})^{[\lambda]} \bdot (y^{2\gamma}x)^{[\theta]} \bdot (y^{2\delta+1}x)^{[\nu]},$$ 
where $\varepsilon, \lambda, \theta, \nu \in \{0,1\}$ and $\varepsilon +  \lambda + \theta + \nu = 3$. 
For $i \in [1,n-2]$, let $U_i = g_{i,1}  \bdot  g_{i,2}$. Since $\phi(g_{i,1} g_{i,2}) \in L$ and $G/L \cong C_2 \oplus C_2$, if $g_{i,1}= y^{a_1}x^{b_1}$ and $g_{i,2}= y^{a_2}x^{b_2}$, then $a_1\equiv a_2 \pmod 2$ and $b_1\equiv b_2\pmod 2$. Let $S_1 = S_L$, $S_2 = S_{yL}$, $S_3 = S_{xL}$, $S_4 = S_{xyL}$, and $m_i = |S_i|$, where $m_1+m_2+m_3+m_4 = 2n-1$ and for exactly one $i \in [1,4]$, say $i_0$, $m_i$ is even. We note that every $h_j$ for $j \in [1,n-2]$ is the product of two terms of $S_i$ for some $i \in [1,4]$.

Suppose that $m_{i_0} > 0$ and let $g'  \bdot  g'' \mid S_{m_{i_0}}$. We may assume without loss of generality that $h_1 = g' \cdot g'' = y^{2b} \in \mathcal F(L)$. Then $g'  \bdot  T_0 \in \mathcal F(L)$ has four terms and it is easy to verify that $y^{2k}, y^{2k+n/2} \in \pi(g'  \bdot  T_0)$ for some $k \in \Z$. Since $h_2 \bdot \dots \bdot  h_{n-2} = (y^{2a})^{[ \frac n2-1]} \bdot (y^{2b})^{[\frac n2-2]}$, Lemma \ref{lem:distinct4prod} ensures that $S$ has an $n$-product-one subsequence. Thus $m_{i_0} = 0$ for some $i_0 \in [1,4]$. 

Let $i \in [1,4] \backslash \{i_0\}$, so that $m_i$ is odd. We claim that $S_i = g^{[m_i]}$ for some term $g$ of $S_i$. This is trivial for $m_i = 1$ and hence we assume $m_i \ge 3$. Let $S_i = g_1  \bdot  \dots  \bdot  g_{m_i}$ and assume to the contrary that $S_i$ has two distinct terms, say $g_1 \neq g_2$. After renumbering if necessary, we assume that $U_j = g_{2j} \bdot  g_{2j+1}$ for every $j \in [1,\lfloor \frac{m_i}{2} \rfloor]$ and $g_1 \mid T_0$. Thus $h_j = g_{2j} \cdot g_{2j+1} \in \pi(U_j) \subset L$. Let $h_1' = g_1 \cdot g_3$. Notice that $g_1 \cdot g_3 \in L$ and $g_1 \cdot g_3 \neq g_2 \cdot g_3$. Since $S$ has no $n$-product-one subsequences, we have that $T = h_1 \bdot \dots \bdot h_{n-2}$ and $T' := h_1' \bdot h_2 \bdot \dots \bdot h_{n-2}$ are both in $\mathcal F(L)$ and have no $\frac n2$-product-one subsequences. Since $\frac n2 \ge 4$, it follows from Lemma \ref{lem:inverseegzcyclic} that $T = T'$, a contradiction. Thus 
\begin{equation}\label{S}
S=(y^{2\alpha})^{[m_1]} \bdot (y^{2\beta+1})^{[m_2]} \bdot  (xy^{2\gamma})^{[m_3]}  \bdot  (xy^{2\delta+1})^{[m_4]},
\end{equation}
where $m_{i_0} = 0$ for some $i_0 \in [1,4]$ and $m_i$ is odd for every $i \in [1,4] \backslash \{i_0\}$.

Let 
$$T_1 = (y^{4\alpha})^{[\lfloor \frac{m_1}{2} \rfloor]}  \bdot  (y^{4\beta+2})^{[\lfloor \frac{m_2}{2} \rfloor]}  \bdot  (y^{4\gamma})^{[\lfloor \frac{m_3}{2} \rfloor]}  \bdot  (y^{4\delta+2})^{[\lfloor \frac{m_4}{2} \rfloor]}$$ 
be a sequence of length $n-2$ in $\mathcal F(L)$.
Since $T_1$ has no $n$-product-one subsequence, it follows from Lemma \ref{lem:inverseegzcyclic} that $\lfloor \frac{m_1}{2} \rfloor + \lfloor \frac{m_3}{2} \rfloor = \lfloor \frac{m_2}{2} \rfloor + \lfloor \frac{m_4}{2} \rfloor = \frac{n}{2}-1$. It suffices to show that $1 \in \{m_1,m_2,m_3,m_4\}$. Assume to the contrary that $m_i \geq 3$ for every $i \in [1,4] \backslash \{i_0\}$. 

For $i_0 \in [1,2]$, let $y^{2k} = xy^{2\gamma} \cdot xy^{2\gamma} \cdot xy^{2\delta+1} \cdot xy^{2\delta+1}$ and $y^{2k'} = xy^{2\gamma} \cdot xy^{2\delta+1} \cdot xy^{2\gamma} \cdot xy^{2\delta+1} = y^{2k+\frac{n}{2}}$.
For $i_0 = 3$, let $y^{2k} = y^{2\beta+1} \cdot y^{2\beta+1} \cdot xy^{2\delta+1} \cdot xy^{2\delta+1}$ and $y^{2k'} = y^{2\beta+1} \cdot xy^{2\delta+1} \cdot y^{2\beta+1} \cdot xy^{2\delta+1} = y^{2k+\frac{n}{2}}$.
For $i_0 = 4$, let $y^{2k} = y^{2\beta+1} \cdot y^{2\beta+1} \cdot xy^{2\gamma} \cdot xy^{2\gamma}$ and $y^{2k'} = y^{2\beta+1} \cdot xy^{2\gamma} \cdot y^{2\beta+1} \cdot xy^{2\gamma} = y^{2k+\frac{n}{2}}$.

In any case, there exist a subsequence $V \mid S$ of length $|V| = 4$ and an integer $k$ such that $y^{2k}, y^{2k + \frac{n}{2}} \in \pi(V)$. 
By Lemma \ref{lem:distinct4prod}, we obtain a contradiction.

Therefore $m_i = 1$ for some $i \in [1,4] \backslash \{i_0\}$. To sum up, since $\lfloor \frac{m_1}{2} \rfloor + \lfloor \frac{m_3}{2} \rfloor = \lfloor \frac{m_2}{2} \rfloor + \lfloor \frac{m_4}{2} \rfloor = \frac{n}{2}-1$ and $m_1+m_2+m_3+m_4 = 2n-1$, we have that $(m_1,m_2,m_3,m_4)$ is one of the following tuples \begin{multicols}{4}
\begin{enumerate}[(i)]
\item $(0,1,n-1,n-1)$, 
\item $(0,n-1,n-1,1)$, 
\item $(1,0,n-1,n-1)$, 
\item $(n-1,0,1,n-1)$, 
\item $(n-1,1,0,n-1)$, 
\item $(n-1,n-1,0,1)$, 
\item $(1,n-1,n-1,0)$, 
\item $(n-1,n-1,1,0)$.
\end{enumerate}
\end{multicols}

By the parities of the exponents of $x$ and $y$, it is easily seen that it is not possible to use three distinct terms in the sequences given by the parameters above in order to form $n$-product-one subsequences. Nevertheless, the sequences given by the parameters (vi) and (viii) are $n$-product-one free if and only if $\gcd\left(2\alpha - (2\beta+1), \frac{n}{2}\right) = 1$. Furthermore, if $(n,s) = (2^t, 2^{t-1}+1)$ for some $t \ge 3$, then it is easy to check that $G = \langle x, xy^v \rangle$, where $v$ is odd, and if $(n,s) = (4n_2, 2n_2+1)$ for some $n_2 \ge 3$ odd, then it is easy to check that $G = \langle x, xy^v \rangle$, where $v$ is odd. Therefore the sequences given by (iv) and (v) are equivalent to the sequences given by (vi) and (viii) when considering other generating set. Moreover, the following identities hold.
\begin{align}\label{identities:cases}
(y^{2\alpha})^{\frac{n}{2}} &= 1, \\
(y^{2\beta+1})^{\frac{n}{2}} &= y^{\frac{n}{2}}, \nonumber \\
(xy^{2\gamma})^{\frac{n}{2}} &= 
\begin{cases}
1 &\text{ if $8 \mid n$}, \\
y^{\frac{n}{2}} &\text{ if $8 \nmid n$},
\end{cases} \nonumber \\
(xy^{2\delta+1})^{\frac{n}{2}} &= 
\begin{cases}
1 &\text{ if either $8 \mid n$ and $s \equiv 3 \!\!\! \pmod 4$ or $8 \nmid n$ and $s \equiv 1 \!\!\! \pmod 4$}, \\
y^{\frac{n}{2}} &\text{ if either $8 \mid n$ and $s \equiv 1 \!\!\! \pmod 4$ or $8 \nmid n$ and $s \equiv 3 \!\!\! \pmod 4$}.
\end{cases} \nonumber 
\end{align}

Let $n_1 = \gcd(s+1,n)$ and $n_2 = n/n_1$, so that $s \equiv -1 \pmod {n_1}$, $s \equiv 1 \pmod {n_2}$, and $\gcd(n_1,n_2) \in \{1,2\}$. We note that $n_1$ is always even. Moreover, if $n$ is not a power of $2$, then $n_1 \ge 4$ and $n_2 \ge 3$.

By the identities in \eqref{identities:cases}, we consider the following cases:
\begin{enumerate}[(I)]
\item Case $8 \nmid n$ and $s \equiv 1 \pmod 4$. We may consider only the sequences given by (i) and (iii). It suffices to verify that $(xy^{2\gamma})^{[n-1]} \bdot (xy^{2\delta+1})^{[n-1]}$ contains an $n$-product-one subsequence. Notice that 
$$(xy^{2\gamma})^{2k} \cdot (xy^{2\delta+1})^{n-2k} = 1$$ 
if and only if 
$$k \equiv 0 \pmod {\frac{\frac{n}{2}}{\gcd(\frac{s+1}{2}(2\gamma-(2\delta+1)), \frac{n}{2})}}.$$ 
If $\gcd(\frac{s+1}{2}(2\gamma-(2\delta+1)), \frac{n}{2}) \neq 1$, then $S$ contains an $n$-product-one subsequence. Suppose that $\gcd(\frac{s+1}{2}(2\gamma-(2\delta+1)), \frac{n}{2}) = 1$. In particular, we must have  $n_1 = 2$, since $\frac{n_1}{2}$ divides both $\frac{s+1}{2}$ and $\frac{n}{2}$, which implies that $\frac{n_1}{2}$ divides $\gcd(\frac{s+1}{2}(2\gamma-(2\delta+1)), \frac{n}{2}) = 1$. Hence $n$ is a power of $2$. Since $n \equiv 0 \pmod 4$ and $8 \nmid n$, it follows that $n = 4$, and then $s \equiv \pm 1 \pmod n$, a contradiction.

\item Case $8 \nmid n$ and $s \equiv 3 \pmod 4$. We may consider only the sequences given by (iv) and (v). It suffices to verify that $(y^{2\alpha})^{[n-1]} \bdot (xy^{2\delta+1})^{[n-1]}$ contains an $n$-product-one subsequence. Notice that 
$$(y^{2\alpha})^{2k} \cdot (xy^{2\delta+1})^{n-2k} = 1$$ 
if and only if 
$$k \equiv 0 \pmod {\frac{\frac{n}{2}}{\gcd(2\alpha - (2\delta+1)\frac{s+1}{2} - \frac{n}{4}, \frac{n}{2})}}.$$ 
If $\gcd(2\alpha - (2\delta+1)\frac{s+1}{2} - \frac{n}{4}, \frac{n}{2}) \neq 1$, then $S$ contains an $n$-product-one subsequence. Suppose that $\gcd(2\alpha - (2\delta+1)\frac{s+1}{2} - \frac{n}{4}, \frac{n}{2}) = 1$. Then 
$$(y^{2\alpha})^{2k-1} \cdot xy^{2\delta+1} \cdot y^{2\alpha} \cdot (xy^{2\delta+1})^{n-2k-1} = 1$$
if and only if 
\begin{equation}\label{cases-iv-v}
\left[2\alpha - (2\delta+1)\frac{s+1}{2} - \frac{n}{4} \right] k \equiv \alpha(1-s) \pmod {\frac{n}{2}},
\end{equation}
which has a solution $k \in [0, \frac{n}{2}-1]$. We are done whether $k \in [1, \frac{n}{2}-1]$. However $k = 0$ yields $2\alpha \equiv 0 \pmod {\frac{n_1}{2}}$. Notice that $4$ divides $n_1 = \gcd(s+1,n)$. Since $\frac{n_1}{4}$ divides $2\alpha$, $\frac{s+1}{2}$ and $\frac{n}{4}$, it follows that $\frac{n_1}{4}$ divides $\gcd(2\alpha - (2\delta+1)\frac{s+1}{2} - \frac{n}{4}, \frac{n}{2}) = 1$, which yields $n_1 = 4$ and $n_2$ is odd. In this case, $s = 2n_2+1$ and $S$ is $n$-product-one free since $G = \langle x, xy^{2\delta+1} \rangle$ and $y^{2\alpha}$ is a power of $xy^{2\delta+1}$ with an even exponent. 

\item Case $8 \mid n$ and $s \equiv 3 \pmod 4$. We may consider only the sequences given by (ii) and (vii). It suffices to verify that $(y^{2\beta+1})^{[n-1]} \bdot (xy^{2\gamma})^{[n-1]}$ contains an $n$-product-one subsequence. Notice that 
$$(y^{2\beta+1})^{2k} \cdot (xy^{2\gamma})^{n-2k} = 1$$ 
if and only if 
$$k \equiv 0 \pmod {\frac{\frac{n}{2}}{\gcd(2\beta+1 - \gamma(s+1) - \frac{n}{4}, \frac{n}{2})}}.$$ 
If $\gcd(2\beta+1 - \gamma(s+1) - \frac{n}{4}, \frac{n}{2}) \neq 1$, then $S$ contains an $n$-product-one subsequence. Suppose that $\gcd(2\beta+1 - \gamma(s+1) - \frac{n}{4}, \frac{n}{2}) = 1$. Then 
$$(y^{2\beta+1})^{2k-1} \cdot xy^{2\gamma} \cdot y^{2\beta+1} \cdot (xy^{2\gamma})^{n-2k-1} = 1$$
if and only if 
\begin{equation}\label{cases-ii-vii}
\left[2\beta+1 - \gamma(s+1) - \frac{n}{4} \right] k \equiv (2\beta+1)\frac{1-s}{2} \pmod {\frac{n}{2}},
\end{equation}
which has a solution $k \in [0, \frac{n}{2}-1]$. We are done whether $k \in [1, \frac{n}{2}-1]$. However $k = 0$ yields $2\beta+1 \equiv 0 \pmod {\frac{n_1}{2}}$. In particular, $\frac{n_1}{2}$ is odd. On the other hand, $s \equiv 3 \pmod 4$ implies that $4$ divides $n_1 = \gcd(s+1,n)$, a contradiction.

\item Case $8 \mid n$ and $s \equiv 1 \pmod 4$. We may consider only the sequences given by (i)-(v) and (vii). We consider three subcases:
\begin{enumerate}[{(IV.}1)]
\item Subcase $(xy^{2\gamma})^{[n-1]} \bdot (xy^{2\delta+1})^{[n-1]} \mid S$, which corresponds to the sequences given by parameters (i) and (iii). Similar to Case (I), we may suppose that $\gcd(\frac{s+1}{2}(2\gamma-(2\delta+1)), \frac{n}{2}) = 1$ and $n$ is a power of $2$, say, $n = 2^t$ for some $t \ge 3$. Thus we have $n_2 = 2^{t-1}$, therefore $s = 2^{t-1}+1$. Notice that 
$$(xy^{2\gamma})^{2k-1} \cdot xy^{2\delta+1} \cdot xy^{2\gamma} \cdot (xy^{2\delta+1})^{n-2k-1} = 1$$
if and only if
$$(2\gamma - (2\delta+1) - 2^{t-2})k \equiv 2^{t-2} \pmod {2^{t-1}},$$
which has a solution $k \in [1,2^{t-1}-1]$ since $2\gamma - (2\delta+1) - 2^{t-2}$ is odd and $k \not\equiv 0 \pmod {2^{t-1}}$. Thus $S$ contains an $n$-product-one sequence.

\item Subcase $(y^{2\alpha})^{[n-1]} \bdot (xy^{2\delta+1})^{[n-1]} \mid S$, which corresponds to the sequences given by parameters (iv) and (v). Similar to Case (II), we may suppose that $\gcd(2\alpha - (2\delta+1)\frac{s+1}{2} - \frac{n}{4}), \frac{n}{2}) = 1$ and $2\alpha \equiv 0 \pmod {\frac{n_1}{2}}$. Notice that $n_1 = \gcd(s+1,n) \equiv 2 \pmod 4$ and then $4 \mid n_2$. Since $\frac{n_1}{2}$ divides $2\alpha$, $\frac{s+1}{2}$ and $\frac{n}{4}$, it follows that $\frac{n_1}{2}$ divides $\gcd(2\alpha - (2\delta+1)\frac{s+1}{2} - \frac{n}{4}, \frac{n}{2}) = 1$, which yields $n_1 = 2$ and $n$ is a power of $2$, say $n = 2^t$ for some $t \ge 3$. In this case, $s = 2^{t-1}+1$ and $S$ is $n$-product-one free since $G = \langle x, xy^{2\delta+1} \rangle$ and $y^{2\alpha}$ is a power of $xy^{2\delta+1}$ with an even exponent.

\item Subcase $(y^{2\beta+1})^{[n-1]} \bdot (xy^{2\gamma})^{[n-1]} \mid S$, which corresponds to the sequences given by parameters (ii) and (vii). Similar to Case (III), we may suppose that $\gcd(2\beta+1 - \gamma(s+1) - \frac{n}{4}), \frac{n}{2}) = 1$ and $\frac{n_1}{2}$ is odd, that is, $n_1 \equiv 2 \pmod 4$. In particular, $4 \mid n_2$. Since $\frac{n_1}{2}$ divides $2\beta+1$, $s+1$ and $\frac{n}{4}$, it follows that $\frac{n_1}{2}$ divides $\gcd(2\beta+1 - \gamma(s+1) - \frac{n}{4}, \frac{n}{2}) = 1$, which yields $n_1 = 2$ and $n$ is a power of $2$, say $n = 2^t$ for some $t \ge 3$. In this case, $s = 2^{t-1}+1$ and Equation \eqref{cases-ii-vii} becomes 
$$(2\beta+1 - 2\gamma - 2^{t-2})k \equiv 2^{t-2} \pmod {2^{t-1}},$$
which has a solution $k \in [1,2^{t-1}-1]$. Therefore, $S$ contains an $n$-product-one subsequence.
\end{enumerate}
\end{enumerate}
This completes the proof of {\it (a)}.

\item Let $S \in \mathcal F(G)$ be a $2n$-product-one free sequence of length $|S| = 3n-1$. Since $|S| > 2n = {\sf s}(G)$, it is possible to extract an $n$-product-one subsequence $T \mid S$. It follows that $|S \bdot T^{[-1]}| = 2n-1$ and $S \bdot T^{[-1]}$ is $n$-product-one free. From item {\it (a)}, it follows that $S \bdot T^{[-1]} = (y^{t_1})^{[n-1]} \bdot (y^{t_2})^{[n-1]} \bdot xy^{t_3}$, where $\gcd(t_1 - t_2, n) = 1$. We claim that either $T = (y^{t_1})^{[n]}$ or $T = (y^{t_2})^{[n]}$. In fact, suppose that there exists $g \mid T$ such that $g \not\in \{y^{t_1} , y^{t_2}\}$. Then $|S \bdot (g \bdot y^{t_1} \bdot y^{t_2})^{[-1]}| = 3n-4 > 2n$, and we may ensure that $S \bdot (g \bdot y^{t_1} \bdot y^{t_2})^{[-1]}$ contains an $n$-product-one subsequence $T'$. Since $g \bdot y^{t_1} \bdot y^{t_2} \mid S \bdot T'^{[-1]}$, we have from item {\it (a)} that $S \bdot T'^{[-1]} = (y^{t_1})^{[n-1]} \bdot (y^{t_2})^{[n-1]} \bdot g$ and hence $g \mid S_{x \langle y \rangle}$. If $g$ is not the same term as $xy^{t_3}$, then $S \bdot (g \bdot xy^{t_3})^{[-1]}$ contains an $n$-product-one subsequence $T''$ and $g \bdot xy^{t_3} \mid S \bdot T''^{[-1]}$. Again from item {\it (a)}, $S \bdot T''^{[-1]}$ is not $n$-product-one free, a contradiction. It implies that $S = (y^{t_1})^{[u]} \bdot (y^{t_2})^{[3n-2-u]} \bdot xy^{t_3}$. Since $\gcd(t_1-t_2, n) = 1$, the only possibilities for an $n$-product-one subsequence are $(y^{t_1})^n = 1$ or $(y^{t_2})^n = 1$. Therefore either $u = 2n-1$ or $u = n-1$, and we are done.

\item Let $S \in \mathcal F(G)$ be a (short) product-one free sequence of length $|S| = n$. Then $S \bdot 1^{[n-1]}$ and $S \bdot 1^{[2n-1]}$ are $n$-product-one free and $2n$-product-one free sequences of lengths $2n-1$ and $3n-1$, respectively. Thus items {\it (a)} and {\it (b)} imply that there exist $\alpha, \beta \in G$ and $t_2, t_3 \in \mathbb Z$ such that $G \cong \langle \alpha, \beta \mid \alpha^2 = \beta^{n/2}, \beta^n = 1, \beta \alpha = \alpha \beta^s \rangle$, $\gcd(t_2, n) = 1$ and $S \bdot 1^{[n-1]} = 1^{[n-1]} \bdot (\beta^{t_2})^{[n-1]} \bdot (\alpha \beta^{t_3})$ and $S \bdot 1^{[2n-1]} = 1^{[2n-1]} \bdot (\beta^{t_2})^{[n-1]} \bdot (\alpha \beta^{t_3})$. Therefore $S = (\beta^{t_2})^{[n-1]} \bdot (\alpha \beta^{t_3})$.
\end{enumerate}
\qed

\section*{Acknowledgements}

This paper was written when the author was a postdoctoral fellow at the University of Graz, Austria.
He would like to thank NAWI Graz for the financial support and all the colleagues in the Algebra and
Number Theory Research Group for their hospitality and the opportunity to learn a lot from the group.
In particular, he would like to thank Qinghai Zhong for discussions during the preparation of this paper.
The author was also partially supported by FAPEMIG grants RED-00133-21, APQ-02546-21 and APQ-
01712-23.


\begin{thebibliography}{99}

\bibitem{AMR} D.V. Avelar, F.E. Brochero Martínez, S. Ribas;
{\em On the direct and inverse zero-sum problems over $C_n \rtimes_s C_2$.} 
J. Combin. Theory Ser. A 197 (2023), 105751.

\bibitem{Bas} J. Bass;
{\em Improving the Erd\H os-Ginzburg-Ziv theorem for some non-abelian groups.} 
J. Number Theory 126 (2007), 217--236.

\bibitem{MLMR} F.E. Brochero Martínez, A. Lemos, B.K. Moriya, S. Ribas; 
{\em The main zero-sum constants over $D_{2n} \times C_2$.} 
SIAM J. Discrete Math. 37(3) (2023), 1496--1508.

\bibitem{MR1} F.E. Brochero Martínez, S. Ribas; 
{\em Extremal product-one free sequences in $C_q \rtimes_s C_m$.}
J. Number Theory 204 (2019), 334--353.

\bibitem{MR2} F.E. Brochero Martínez, S. Ribas; 
{\em Extremal product-one free sequences in dihedral and dicyclic groups.} 
Discrete Math. 341(2) (2018), 570--578.

\bibitem{MR3} F.E. Brochero Martínez, S. Ribas; 
{\em Extremal product-one free sequences over $C_n \rtimes_s C_2$.}
Discrete Math. 345(12) (2022), 113062.

\bibitem{Car} Y. Caro; 
{\em Zero-sum problems - a survey.} 
Discrete Math. 152 (1996), 93--113.

\bibitem{EGZ} P. Erd\H os, A. Ginzburg, A. Ziv; 
{\em Theorem in the additive number theory.} 
Bull. Res. Council Israel 10 (1961), 41--43.

\bibitem{vEBK} P. van Emde Boas, D. Kruyswijk; 
{\em A combinatorial problem on finite abelian groups III.} 
Report ZW-1969-008, Math. Centre, Amsterdam, 1969.

\bibitem{FZ} V. Fadinger, Q. Zhong; 
{\em On product-one sequences over subsets of groups.} 
Period. Math. Hung. 86 (2023), 454--494.

\bibitem{Ga2} W.D. Gao; 
{\em A combinatorial problem on finite abelian groups.} 
J. Number Theory 58 (1996), 100--103.


\bibitem{Ga1} W.D. Gao; 
{\em An addition theorem for finite cyclic groups.} 
Discrete Math. 163 (1997), 257--265.

\bibitem{Ga2} W.D. Gao; 
{\em  On zero-sum subsequences of restricted size II.} 
Discrete Math. 271 (2003), 51--59.


\bibitem{GaGe} W.D. Gao, A. Geroldinger; 
{\em Zero-sum problems in finite abelian groups: a survey.} 
Expo. Math. 24 (2006), 337--369.

\bibitem{GLQ} W. Gao, Y. Li, Y. Qu; 
{\em On the invariant ${\sf E}(G)$ for groups of odd order.} 
Acta Arith. 201 (2021), 255-267.

\bibitem{GGOZ} A. Geroldinger, D.J. Grynkiewicz, J. Oh, Q. Zhong; 
{\em On product-one sequences over dihedral groups.} 
J. Algebra Appl. 21 (2022), 2250064.

\bibitem{GeHK} A. Geroldinger, F. Halter-Koch; 
{\em Non-unique factorizations: algebraic, combinatorial and analytic theory.} 
Pure and Applied Mathematics 278, Chapman \& Hall/CRC, 2006.

\bibitem{HZ} D.C. Han, H.B. Zhang; 
{\em Erd\H os-Ginzburg-Ziv theorem and Noether number for $C_m \ltimes_\varphi C_{mn}$.} 
J. Number Theory 198 (2019), 159--175.

\bibitem{OZ1} J.S. Oh, Q. Zhong; 
{\em On Erd\H os-Ginzburg-Ziv inverse theorems for dihedral and dicyclic groups.}
Israel J. Math 238 (2020), 715--743.

\bibitem{Ols1} J.E. Olson; 
{\em A combinatorial problem on finite Abelian groups I.} 
J. Number Theory 1 (1969), 8--10.

\bibitem{Ols2} J.E. Olson; 
{\em A combinatorial problem on finite Abelian groups II.} 
J. Number Theory 1 (1969), 195--199.


\bibitem{QL1} Y. Qu, Y. Li; 
{\em Extremal product-one free sequences and $|G|$-product-one free sequences of a meta-cyclic group.} 
Discrete Math. 345(8) (2022), 112938.

\bibitem{QL2} Y. Qu, Y. Li; 
{\em On a conjecture of Zhuang and Gao.} 
Colloq. Math. 171 (2023), 113--126.

\bibitem{Yu} T. Yuster; 
{\em Bounds for counter-example to addition theorem in solvable groups.} 
Arch. Math. 51 (1988), 223--231.

\bibitem{YP} T. Yuster, B. Petersen; 
{\em A generalization of an addition theorem for solvable groups.} 
Can. J. Math. 36(3) (1984), 529--536.

\bibitem{ZhGa} J.J. Zhuang, W.D. Gao; 
{\em Erd\H os-Ginzburg-Ziv theorem for dihedral groups of large prime index.} 
Europ. J. Combin. 26 (2005), 1053--1059.
\end{thebibliography}
\end{document}